\documentclass{amsart}
\usepackage{amssymb,latexsym}
\usepackage{newlattice}

\DeclareMathOperator{\Princ}{Princ}
\DeclareMathOperator{\Frame}{Frame}
\DeclareMathOperator{\Prime}{Prime}
\DeclareMathOperator{\Flag}{Flag}
\DeclareMathOperator{\col}{col}

\DeclareMathOperator{\nat}{nat}
\DeclareMathOperator{\colset }{colSet}

\DeclareMathOperator{\Rep}{Rep}

\DeclareMathOperator{\CompLat}{CompLat}

\newcommand{\Jp}[1]{\tup{J}^+(#1)}

\newcommand{\gsum}{\overset{\text{\large \tbf .}}{+}} 

\newtheorem{theorem}{Theorem}
\newtheorem{lemma}[theorem]{Lemma}
\newtheorem{corollary}[theorem]{Corollary}
\newtheorem{definition}[theorem]{Definition}

\theoremstyle{definition} 
\newtheorem*{problem}{Problem}

\begin{document}
\title[Characterizing representability]
{Characterizing representability\\ 
by principal congruences\\
for finite distributive lattices\\
with a join-irreducible unit element}  
\author{George Gr\"{a}tzer} 
\email[G. Gr\"{a}tzer]{gratzer@me.com}
\urladdr[G. Gr\"{a}tzer]{http://server.math.umanitoba.ca/~gratzer/}

\date{\today}

\begin{abstract}
For a finite distributive lattice $D$,  
let us call $Q \ci D$
\emph{principal congruence representable}, 
if there is a finite lattice $L$
such that the congruence lattice of $L$ 
is isomorphic to $D$ and the principal congruences of $L$ 
correspond to $Q$ under this isomorphism.

We find a necessary condition for representability
by principal congruences
and prove that for finite distributive lattices
with a join-irreducible unit element this condition is also sufficient. 
\end{abstract}

\maketitle

\section{Introduction}\label{S:Introduction}

\subsection{Background}\label{S:Background}
For a finite lattice $L$,
we denote by $\Con L$ the congruence lattice of $L$,
by $\Princ L$ the ordered set of principal congruences of $L$,
and by $\Prime L$ the set of prime intervals of $L$. 
Let $\J L$ denote the (ordered) set of 
join-irreducible elements of $L$, and let 
\begin{equation}\label{E:ji}
   \Jp L = \set{0,1} \uu \J L.
\end{equation}
Then for a finite lattice $L$,
\begin{equation}\label{E:cont}
   \Jp {\Con L} \ci \Princ L \ci \Con L,
\end{equation}
since every join-irreducible congruence is generated by a prime interval; furthermore, $\zero = \con{x,x}$ for any $x \in L$ and $\one = \con{0,1}$.

This paper continues G.~Gr\"atzer \cite{gG14}
(see also \cite[Section 10-6]{LTS1} and \cite[Part VI]{CFL2}),
whose main result is the following statement. 

\begin{theorem}\label{T:bounded}
Let $P$ be a bounded ordered set.
Then there is a bounded lattice~$K$ such that $P \iso \Princ K$.
If the ordered set $P$ is finite, 
then the lattice $K$ can be chosen to be finite.
\end{theorem}

The bibliography lists a number of papers related to this result. 

In G. Gr\"atzer and H. Lakser~\cite{GLa}, 
we got some preliminary results 
for the following related problem.

For a finite distributive lattice $D$,  
let us call $Q \ci D$
\emph{principal congruence representable}
(\emph{representable}, for short), 
if there is a finite lattice $L$
such that $\Con L$
is isomorphic to $D$ and $\Princ L$ 
corresponds to $Q$ under this isomorphism.
Note that by \eqref{E:ji} and \eqref{E:cont},
if $Q$ is representable, then $\J D \ci Q$ and $0, 1 \in Q$. 

We now state \cite[Problem 22.1]{CFL2}.

\begin{problem}\label{P:main}
Characterize representable sets for finite distributive lattices.
\end{problem}
 
In this paper, we investigate a combinatorial condition 
for representability.
We prove that this condition is necessary, and for 
finite distributive lattices with a~join-irreducible unit,
it is also sufficient.

\subsection{Chain representability}\label{S:repr}

A finite chain $C$ is \emph{colored by}
an ordered set $P$, 
if~there is a map $\col$
of $\Prime C$ \emph{onto} $P$. 
For $x \leq y \in C$, define the \emph{color set}, 
denoted by $\colset [x, y]$, of an interval $[x, y]$ of $C$ 
as the set of colors of the prime intervals~$\fp$ in $[x, y]$; in formula,
\begin{equation}\label{E:col}
   \colset [x, y] = \setm{\col\fp}{\fp \in \Prime [x, y]}.
\end{equation}

Note that if $\col \fp = p \in P$, then $\colset  \fp = \set{p}$. 

We use Figure~\ref{F:Repr} to illustrate coloring. 
Figure~\ref{F:Repr} shows an ordered set~$P$ and a~chain 
$C = \set{c_1 \prec c_2 \prec c_3 \prec c_4 \prec c_5}$ 
colored by $P$; the color of a prime interval
is indicated by a label.
Let $D$ be a distributive lattice  
with~$P$ as the ordered set $\J D$ of join-irreducible elements
of~$D$; the elements of $\J D$ 
are gray-filled in~the diagram of $D$ in~Figure~\ref{F:Repr}.

For the color sets of prime intervals of this chain $C$, 
we obtain $\set{a}$, $\set{b}$, and $\set{d}$.
Intervals of length $2$ of $C$ produce two more color sets
$\set{a, b}, \set{b, d}$; 
for instance, $\colset [c_1, c_3] = \set{a, b}$ and 
$\colset [c_2, c_4] = \set{b, d}$. 
There are two intervals of length~$3$, 
but only $[c_1, c_4]$ yields a new color set:
$\colset [c_1, c_4] = \set{a, b, d}$. Finally,
$C = [c_1, c_5]$ gives the same color set as $[c_1, c_4]$.

While the map $\colset$ assigns a color set to an interval of $C$,
the map $\Rep$ assigns an element of $D$ to an interval of $C$:
\[
   \Rep \colon [x, y] \mapsto \JJ\colset [x, y].
\]
Note that $\colset  [x, x] = \es$, so $\Rep[c_1, c_1] = o$;
also, $\Rep[c_1, c_2] = a$, and so on,
as illustrated in Figure~\ref{F:Repr}.

\begin{figure}[b!]
\centerline{\includegraphics[scale=1]{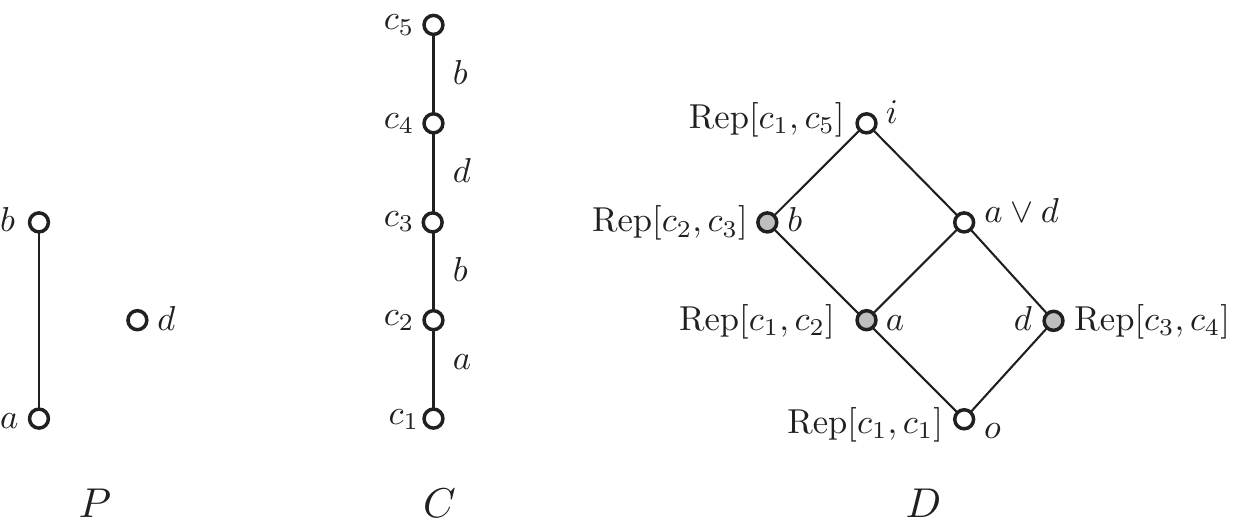}}
\caption{Coloring and representation}\label{F:Repr}
\end{figure}

The set 
\begin{equation}
   \Rep C = \setm{\Rep[x, y]}{x \leq y \in C}
\end{equation}
is a subset of $D$; in our example,
\[
   \Rep C = \set{o, a, b, d, b \jj d} = D - \set{a \jj d}.
\]
Such a subset of $D$ we call \emph{\lp colored-\rp{}chain representable}.
Note that the chain $C$ is not directly related 
to the finite lattice $L$ representing $D$ as a congruence lattice.

Now we state our first result.

\begin{theorem}\label{T:main}
Let $D$ be a finite distributive lattice and let $Q \ci D$.
If $Q$ is representable, then it is chain representable. 
\end{theorem}

\subsection{Finite distributive lattices 
with join-irreducible units}

Let $D$ be a finite distributive lattice
with a join-irreducible unit. 
Then we can apply to $D$ the construction in my paper \cite{gG14}
to obtain a finite lattice $L$ with very special properties
whose congruence lattice is isomorphic to $D$.
Using this as a starting point, we prove our second result.

\begin{theorem}\label{T:main2}
Let $D$ be a finite distributive lattice with a join-irreducible
unit element. Let $Q \ci D$.
Then $Q$ is representable if{}f it is chain representable. 
\end{theorem}

\subsection{Outline} In Section~\ref{S:representable}, 
we prove Theorem~\ref{T:main}.
Section~\ref{S:PbP} provides a \emph{Proof-by-Picture}
of Theorem~\ref{T:main2}. In view of Theorem~\ref{T:main},
to prove Theorem~\ref{T:main2}, 
it is sufficient to verify 
that if $Q \ci D$ is chain representable, 
then there is a lattice $L$ representing it.
This lattice $L$ is constructed in Section~\ref{S:Construction}.
Finally, in Section~\ref{S:main2}, we prove Theorem~\ref{T:main2}.
Section~\ref{S:Discussion} concludes the paper 
with a discussion of some very recent results. 

I would like to thank the referee 
for the many improvements he recommended.

\subsection{Notation}
We use the notation as in \cite{CFL2}.
You can find the complete

\emph{Part I. A Brief Introduction to Lattices} and  
\emph{Glossary of Notation}

\noindent of \cite{CFL2} at 

\verb+tinyurl.com/lattices101+

\section{Proving Theorem~\ref{T:main}}\label{S:representable}

Let $D$ be a finite distributive lattice and let $Q \ci D$.
Let $P$ denote the set of join-irreducible elements of $D$.
Finally, let $Q$ be representable by a finite lattice~$L$,
with bounds $0$ and $1$, 
so $\Con L = D$ and $Q$ is the set of principal congruences
of $L$.

If $C$ is a maximal chain of $D$, then $C$ has a \emph{natural coloring in $L$}: $\nat_L \fp = \con{\fp}$.

Let $K$ be a finite lattice and let $C$ be a maximal chain of $K$.
Then 
\[
   \nat_K(\Prime C) \ci \nat_K(\Prime K);
\] 
note that $\nat_K (\Prime C)$ is a proper subset, in general. 
However,
\begin{equation}\label{E:prime}
   \JJ \nat_K (\Prime K) = \JJ \nat_K (\Prime C) = \one.
\end{equation}

Let $P_1$ and $P_2$ be ordered sets. 
Recall that $P_1 + P_2$ denotes the (ordinal) \emph{sum} 
of $P_1$ and $P_2$ ($P_2$ on top of $P_1$).  
If $P_1$ has a unit, $1_{P_1}$, 
and $P_2$ has a zero, $0_{P_2}$, then 
we obtain the \emph{glued sum} $P_1 \gsum P_2$ 
from $P_1 + P_2$ by identifying $1_{P_1}$ and $0_{P_2}$.

Now to prove Theorem~\ref{T:main}, 
we enumerate all maximal chains of the lattice $L$: 
$C_1, C_2, \dots, C_m$.
Let 
\[
   C_i'= 
   \begin{cases}
      C_i \text{\q for $i$ odd;}\\
      \widetilde C_i \text{\q for $i$ even,}
   \end{cases}
\]
where $\widetilde C_i$ denotes the dual of $C_i$ and define
the chain $C$ as a glued sum:
\[
   C = C_1' \gsum \dots \gsum C_m'.
\]

We have a coloring $\col$ for $C$: 
if $\fp$ is a prime interval in $C$, 
then $\fp$ is a~prime interval in exactly one $C_i'$. 
Let $\fp' = \fp$ if $i$ is odd and let $\fp' = \tilde\fp$ 
be the dual of $\fp$ if $i$ is even.  
Since $C_i$ is a maximal chain in $L$, it follows that
$\fp'$ is a prime interval in~$L$.
Then $\col \fp = \con{\fp'} \in P$ 
defines a coloring of $C$.
So for $[x,y] \ci C$, we obtain the color set
$\colset [x, y] = \setm{\col \fp}{\fp \in \Prime [x, y]}$.
We~define
\begin{equation}
   \Rep C = \setm{\JJ(\colset [x,y])}{[x,y] \ci C},
\end{equation}
a subset of $D$.

To prove Theorem~\ref{T:main}, we have to establish that $\Rep C = Q$.

To verify that $\Rep C \ce Q$, 
let $x \in Q$. 
By the definition of $D$, $Q$, and~$L$, 
we can represent $x$ as a principal congruence
$\con{a,b}$ in~$L$ for some $a \leq b \in L$. 
Let $C_i$ be one of the maximal chains in $L$ with $a, b \in C_i$.
Applying \eqref{E:prime} to the interval $[a, b]$ of~$L$, 
we obtain that
\[
   \JJ(\colset_{C_i} [a,b])= \JJ(\colset_{C} [a,b]) = x.
\]  
Therefore, $x \in \Rep C$.

Conversely, to verify that $\Rep C \ci Q$,
let $x = \Rep C$.
Then there are $u \leq v \in C$ 
such that $x =\JJ(\colset_{C_i} [u,v])$.
We distinguish three cases.

Case 1. $u \leq v \in C_i$ for some $1 \leq i \leq n$.
This is easy, just like the converse case, 
utilizing \eqref{E:prime}.

Case 2.  $u \in C_i'$, $v \in C_j'$ for $1 \leq i + 1 < j \leq n$.
In this case, $[u, v] \ce C_{i+1}'$
and so~$x = \one \in \Rep C$.

Case 3. $u \in C_i'$, $v \in C_{i+1}'$ for some $1 \leq i < n$.
Without loss of generality, 
we can assume that $i$ is odd, 
so $C_i' = C_i$ and $C_{i+1}' = \widetilde C_{i+1}$.
Then 
\begin{align*}
x = \JJ(\colset_{C} [u,v]) &= 
     \JJ(\colset_{C_i'} [u,1_{C_i'}]) \jj 
     \JJ(\colset_{C_{i+1}'} [0_{C_{i+1}'},v]) \\
      &= \JJ(\colset_{C_i'} [u,1]) \jj
     \JJ(\colset_{C_{i+1}'} [1,v]) \\      
      &=\con{u, 1} \jj \con{ v, 1}
      = \con{u \mm v, 1} \in Q,
\end{align*}
which we wanted.

This completes the proof of Theorem~\ref{T:main}.

\section{Finite distributive lattices 
with join-irreducible units\\\textsc{``Proof-by-Picture''}}
\label{S:PbP}
\subsection{A colored chain}

Recall that, as in \cite{CFL2},
a \emph{Proof-by-Picture} is not a proof, 
just an illustration of an idea.
We illustrate the proof of Theorem~\ref{T:main2}
with the chain~$C$ 
colored by the ordered set $P = \set{p,q,r,1}$
and the distributive lattice $D$
with a join-irreducible unit satisfying $\J D = P$,
see Figure~\ref{F:PbP}. Note that $1 \in P$;
let $P^* = P-\set{1}$.

\begin{figure}[b!]
\centerline{\includegraphics[scale = 1]{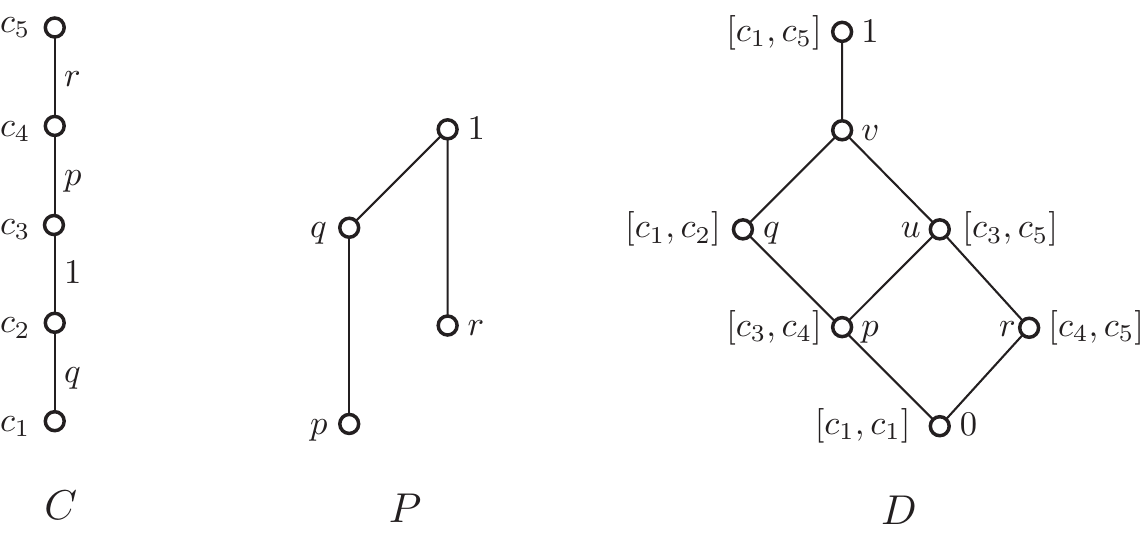}}
\caption{Illustrating the proof of Theorem~\ref{T:main2}}\label{F:PbP}
\end{figure}

As in Figure~\ref{F:Repr}, we mark an element $z \in D$
with the interval $[x, y]$ of $C$, if 
\[
   z = \Rep[x, y] = \JJ\colset[x, y],
\] 
that is, if the element $z \in D$ 
is the join of the colors in $[x, y]$. 
All the elements thus marked form the set $Q = D - \set{v}$.
By definition, $Q$ is chain representable.

We will outline how to construct a finite lattice $L$ 
such that $\Con L$ is isomorphic to~$D$ and $\Princ L$ 
corresponds to $Q$ under this isomorphism, that is,
$Q$ is representable.

For a finite lattice $K$ with zero, $o$ and unit, $i$,
we denote by $K^-$ the ordered set obtained 
by deleting the elements $o$ and $i$ from $K$.

\subsection{The frame lattice}\label{S:Frame}

For the chain $C$ colored by the ordered set $P$, 
see Figure~\ref{F:PbP},
we first construct the \emph{frame lattice} of $C$, $\Frame C$, 
as illustrated in Figure~\ref{F:Frame}, 
consisting of the following elements:
\begin{enumerate}
\item the elements $o$, $i$, 
the zero and unit of $\Frame C$, respectively;
\item the elements $a_p < b_p$ for every $p \in P$;
\item an element $s_1$, a sectional complement of $a_1$ in $b_1$,
that is, $s_1 \mm a_1 = o$ and $s_1 \jj a_1 = b_1$;
\item the chain $C$;
\item a universal complement $u$, that is, $u \mm x = o$
and $u \jj x = i$ for every $c \in (\Frame C)^-$.
\end{enumerate}

These elements are ordered and
the lattice operations are formed as in Figure~\ref{F:Frame}
(which shows the construction 
for the colored chain $C$ of Figure~\ref{F:PbP}). 

Note that $\Frame C$ is a union of $\set{0,1}$-sublattices:
the chains $C_p = \set{o, a_p, b_p, i}$, for $p \in P$, $C$, 
the chain $C_u = \set{o, u, i}$,
and the additional nonchain $\set{0,1}$-sublattice 
$S = \set{o, a_1, b_1, s_1, i}$.

The frame lattice $\Frame C$ in this paper is based on the idea 
of the frame lattice in G.~Gr\"atzer \cite{gG14};
the details are different, 
especially, the inclusion of the chain $C$.

\begin{figure}[htb]
\centerline{\includegraphics[scale=0.60]{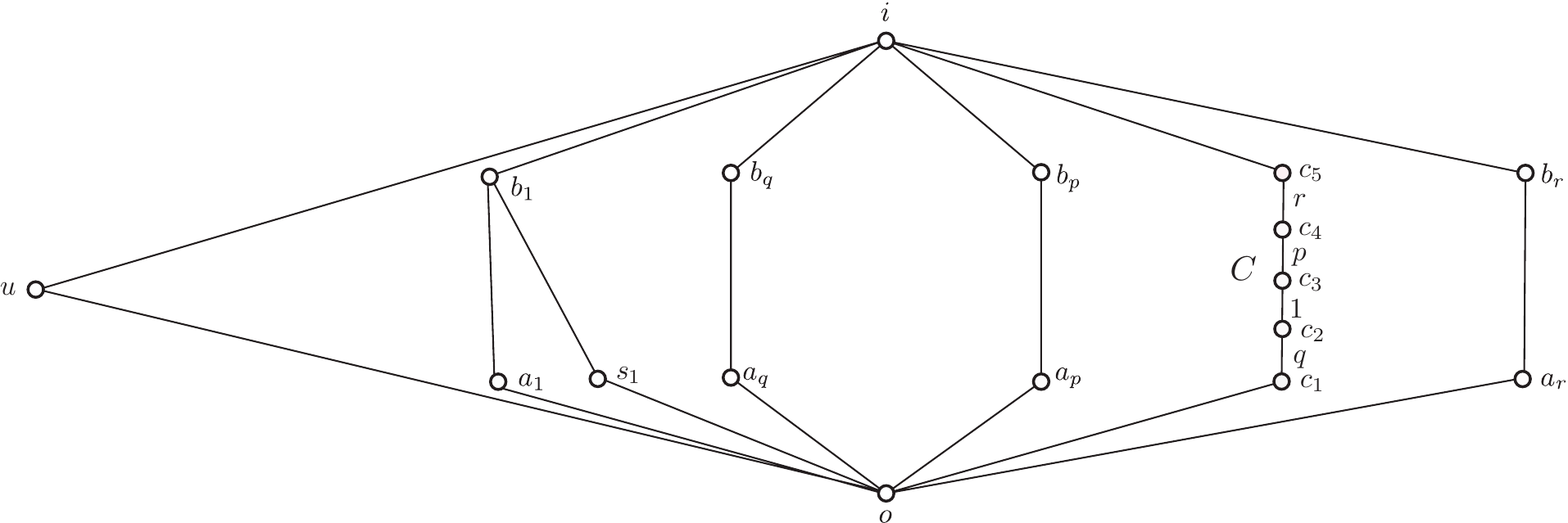}}
\caption{The frame lattice $\Frame C$
for the colored chain $C$ of Figure~\ref{F:PbP}}
\label{F:Frame}
\end{figure}

\begin{figure}[htb]
\centerline{\includegraphics[scale=0.9]{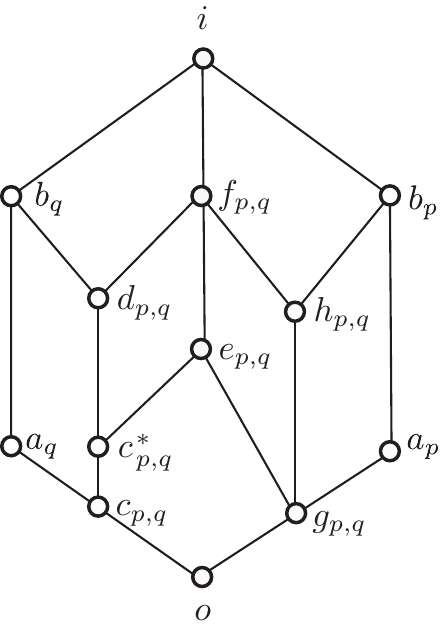}}
\caption{The lattice $W(p, q)$ for $p < q \in P$}\label{F:W}

\bigskip

\bigskip

\centerline{\includegraphics[scale=.65]{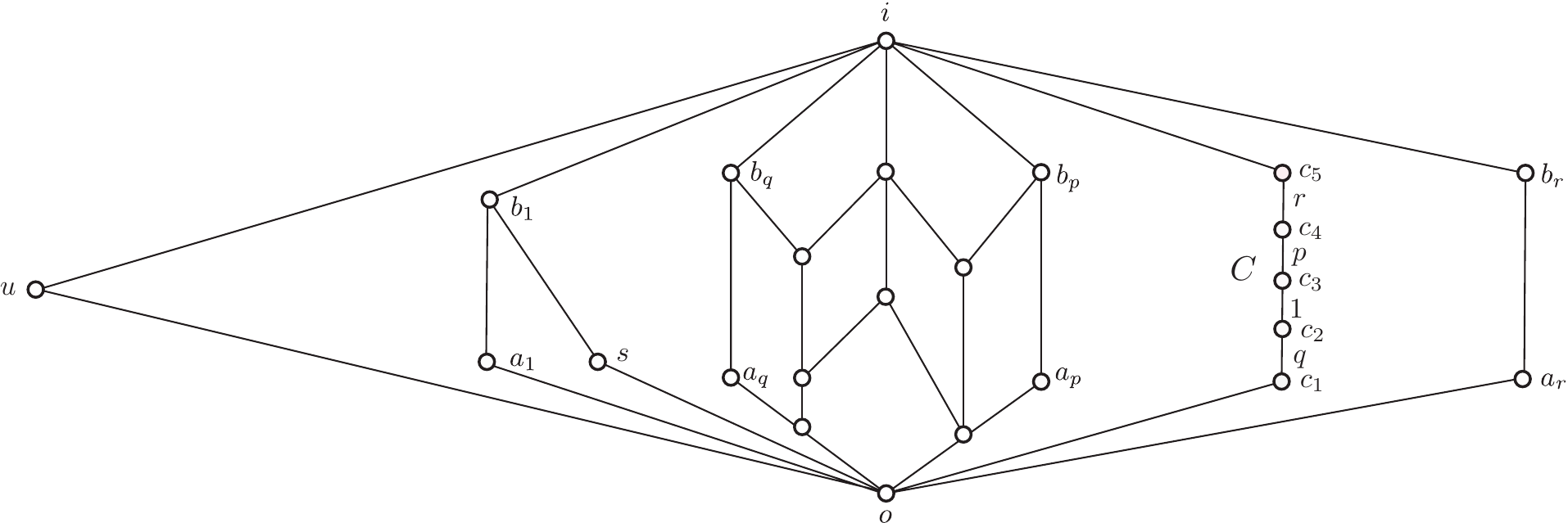}}
\caption{Adding $W(p,q)$ to $\Frame C$ for $p < q \in P$ with $q < 1$}\label{F:Frame+W}
\end{figure}

\subsection{The ordered set W}

We are going to construct the lattice $L$ 
representing $Q \ci D$ 
as an extension of the frame lattice of $C$, $\Frame C$. 
The principal congruence $\con{a_p, b_p}$ of $L$ 
represents $p \in P$.

We use the lattice $W(p,q)$, for $p < q \in P$,
see Figure~\ref{F:W}.
We~add these as sublattices to extend~$\Frame C$.

The lattice $W(p,q)$ is a variant of the lattice $S(p, q)$
in my paper \cite{gG14}. The lattice $W(p,q)$ has two more elements than $S(p, q)$, but from a technical point of~view it is much easier to work with. 
For instance, the crucial formula \eqref{E:Lequ} 
does not hold if we utilize the lattices $S(p, q)$.

Since $p < q$ in $P$, 
we want $\con{a_p, b_p} < \con{a_q, b_q}$
to hold in the extended lattice.
We add seven elements to the sublattice $C_p \uu C_q$
of $\Frame C$, as illustrated in Figure~\ref{F:Frame+W},
to form the sublattice $W(p,q)$. 
This will ensure that $\con{a_p, b_p} \leq \con{a_q, b_q}$.

\subsection{Flag lattices}\label{S:flag}
Figure~\ref{F:Flag} shows a flag lattice $\Flag(c_3)$,
where $C$ is the colored chain of Figure~\ref{F:PbP}
and $\col [c_3, c_4] = p$.

We add eight elements (black filled in Figure \ref{F:Flag}) 
to the sublattice $\set{o, a_p, b_p, c_3, c_4, i}$
of $\Frame C$, as illustrated in Figure \ref{F:AddingFlag},
to form the sublattice $\Flag(c_i)$. 
This extension ensures that $\con{a_p, b_p} = \con{c_3, c_4}$,
where $\col [c_3, c_4] = p$.
Note that if $i \neq j$, then $\Flag(c_i)\ii\Flag(c_j) = C \uu \set{o,i}$.

Similarly, we add $\Flag(c_1), \Flag(c_2), \Flag(c_4)$
to form $L$. We will not draw this extension 
because even with the diagram the resulting lattice
is hard to visualize.

\begin{figure}[hbt]
\centerline{\includegraphics[scale=0.86]{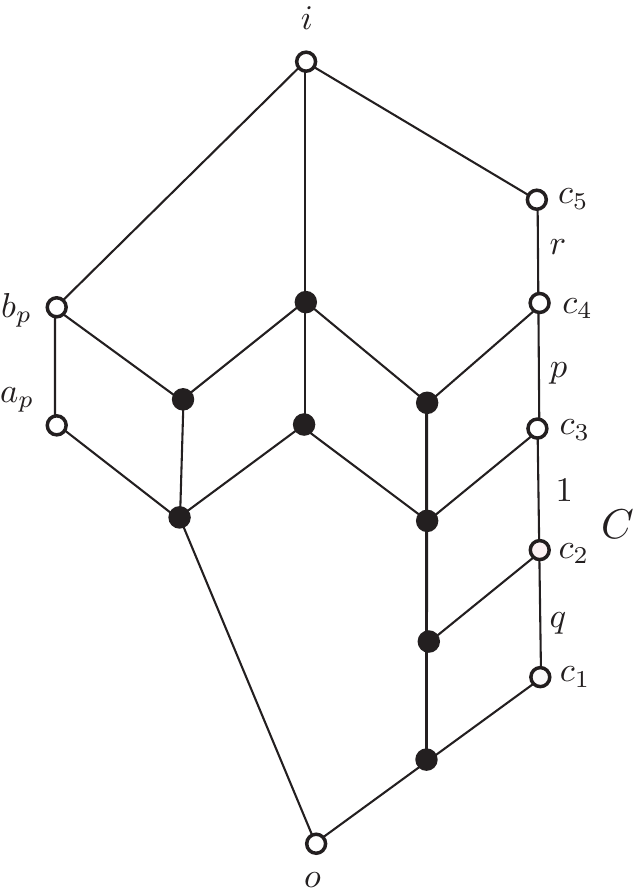}}
\caption{The lattice $\Flag(c_3)$}\label{F:Flag}
\end{figure}

\begin{figure}[htp]
\centerline{\includegraphics[scale=.65]{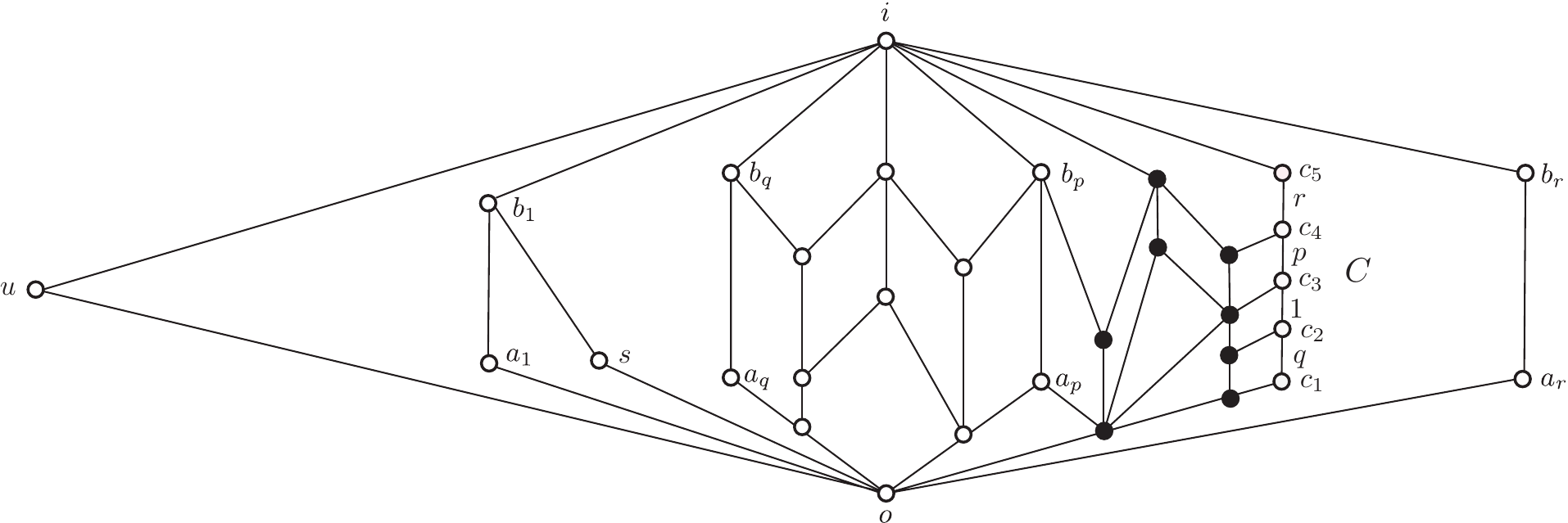}}
\caption{Further adding $\Flag(c_3)$}\label{F:AddingFlag}
\end{figure}

\subsection{The role of $C$}
It follows that $\con{a_p, b_p} \jj \con{a_r, b_r}$
is principal in $L$. Indeed
\[
   \con{a_p, b_p} \jj \con{a_r, b_r} = \con{c_3, c_5}.
\]

On the other hand, $\con{a_p, b_p} \jj \con{a_q, b_q}$ 
is not principal since there is no interval $[x,y]$ in $C$ 
such that  $\colset[x,y] = \set{p,q}$.

\section{Construction}\label{S:Construction}

Let $D$ be a finite distributive lattice 
with a join-irreducible unit element and let $P = \J D$. 
We can assume that $|D| > 2$, 
because Theorem~\ref{T:main2} is trivial if $|D| \leq 2$.
Let $Q \ci D$ be representable 
by the chain $C = \set{c_1 \prec c_2 \prec \dots \prec c_n}$ 
colored by $P$. 
Note that $1 \in P$, so there is at least one prime interval in $C$ colored by $1$.

In this section, we construct a finite lattice $L$ such that
$\Con L$ is isomorphic to $D$ and $\Princ L$
corresponds to $Q$ under this isomorphism,
as required by Theorem~\ref{T:main2}. 
In view of Theorem~\ref{T:main}, this construction of the lattice $L$ and the verification of its properties in Section~\ref{S:main2},
will complete the proof of Theorem~\ref{T:main2}.  

\subsection{The frame lattice}\label{S:Frame C}
As in Section~\ref{S:Frame}, 
we first construct the lattice $\Frame C$, 
see Figure~\ref{F:Frame}.
Also recall that $\Frame C$ is the union of 
the chains $C_p = \set{o, a_p, b_p, i}$, for $p \in P$, 
the chain $C$, 
and the chain $C_u = \set{o, u, i}$ 
with an additional nonchain sublattice, 
$S = \set{o, a_1, b_1, s_1, i}$. 
In formula, 
\begin{equation}
   \Frame C = C \uu S \uu C_u \uu \UUm{C_p}{p \in P},
\end{equation}
where any two distinct components intersect in $\set{o,i}$.

\subsection{The lattice $L$}\label{S:L}
We are going to construct the lattice $L$ 
(of Theorem~\ref{T:main2})
as an extension of $\Frame C$. 

We utilize the following lattices, which we shall call \emph{component lattices}:
\begin{align}
   &W(p,q) &&\text{ for }p < q \in P;\label{E:W}\\
   &\Flag(c_i) &&\text{ for } i < n;\label{E:Flag}\\
   &S,\label{E:S}\\
   &C_u.\label{E:u}
\end{align}
Let $\CompLat C$ be the set of component lattices
associated with the colored chain~$C$.

Recall from Section~\ref{S:flag} that $\Flag(c_i)\ii\Flag(c_j) = C \uu \set{o,i}$ for $i \neq j$.

We start with a simple, but crucial, observation.

\begin{lemma}\label{L:meetirr}
Let $A \neq B \in \CompLat C$.
If either $A$ or $B$ is not a flag lattice,
then $A \ii B$ is a chain $X$: $\set{o,i}$,
$C_p$, or $C \uu \set{o,i}$. 
Moreover, the elements of $X - \set{o}$ are meet-irreducile.
\end{lemma}

We define the set
\begin{align}\label{E:L}
   L = \UUm{W(p,q)}{p < q \in P} 
   \uu \UUm{\Flag(c_i)}{i< n} \uu S \uu C_u.
\end{align}

Define the order relation $\leq$ on $L$ as follows:

$x \leq y$ in $L$ if{}f $x \leq y$ holds 
in one of the component lattices.

In formula,
\begin{equation}\label{E:Lequ}
   \leq = \UUm{\leq_{W(p,q)}}{p < q \in P}
      \uu \UUm{\leq_{\Flag(c_i)}}{i< n} 
       \uu \leq_{S} \uu\leq_{C_{u}}\!\!.
\end{equation}

\begin{lemma}
The binary relation $\leq$ on $L$ is an order relation.
\end{lemma}

\begin{proof}
By \eqref{E:L} and \eqref{E:Lequ}, the relation $\leq$ is reflexive; by definition, it is antisymmetric.

Let $x \leq y \leq z$ in $L$. If $x = y$ or $y = z$,
then $x \leq z$ trivially holds, so we can assume that $x < y < z$.
By the definition of $\leq$ in $L$, see \eqref{E:Lequ}, 
there are component lattices $A$ and $B$ 
so that $x < y$ in $A$ and $y < z$ in $B$.

If $A = B$, then $x < z$ in $A$, therefore, $x < z$ in $L$.
So we can assume that $A$ and~$B$ are distinct lattices
and Lemma~\ref{L:meetirr} applies.
It follows that $y = a_p$ or $y = b_p$ for some $p \in P$. 

If $y = a_p$, then $y < z$ in $B$ 
implies that $b_p = y^* \leq z$ in $B$. 
So $z$ is $b_p$ or $i$, and $x < z$ in $A$ and, therefore,
$x < z$ in $L$ follows.

If $y = b_p$, then again $b_p  < z$ in $B$ and 
$x < z$ in $A$ and, therefore, $x < z$ in $L$ follows.
\end{proof}

\begin{corollary}
The ordered set $L$ is a lattice and 
each component lattice is a sublattice.
\end{corollary}

In fact, the union of any number 
of component lattices is a sublattice.

Let $A \neq B \in \CompLat C$. 
We call them \emph{adjacent}, if $A \ii B \neq \set{o,i}$.

\begin{corollary}\label{C:compl}
Let $A$ and $B$ be not adjacent component lattices.
Then $a \in A - \set{o,i}$ and $b \in B - \set{o,i}$
are complementary.
\end{corollary}

Let $U$ be a $\set{o,i}$-chain in $L$. 
Then for every $x \in L$,
there is a smallest element $x^U \geq x$ of $U$ and
and a largest element $x_U \leq x$ of $U$.

If $A$ and $B$ are adjacent component lattices, 
then $U(A, B) = A \ii B$ is a $\set{o,i}$-chain of $L$.

\begin{corollary}\label{C:adjacent}
Let $A$ and $B$ be adjacent component lattices.
Let $a \in A - \set{o,i}$ and $b \in B - \set{o,i}$.
Then $a \jj b = \max (a^U, b^U)$, where $U = U(A,B)$.
\end{corollary}

\section{Proving Theorem~\ref{T:main2}}\label{S:main2}

\subsection{Two lemmas}\label{S:lemmas}
In this section, under the same assumptions 
as in Section~\ref{S:Construction}, 
we describe the congruences of the lattice $L$ 
we constructed in the previous section.
We verify that the finite distributive lattice 
$D$ is isomorphic to the congruence lattice of $L$
and under this isomorphism, the elements of $Q$
correspond to the principal congruences.

We start the proof with two easy lemmas.

\begin{lemma}\label{L:big}
For every $x \in L$, there is an 
$\set{o,i}$-sublattice $A$ of $L$ containing~$x$ and 
isomorphic to $\SM 3$.
\end{lemma}

\begin{proof}
For $x \in \set{o,i,u, a_1}$, take $A = \set{u, a_1, c_1, o, i}$.
If $x  = s$, then $A = \set{u, s, c_1, o, i}$ is such a sublattice.
Otherwise, let $A = \set{u, a_1, x, o, i}$.
\end{proof}

An~\emph{internal congruence} of $L$
is a congruence $\bga > \zero$, 
such that $\set{o}$ and $\set{i}$ are 
congruence blocks of~\bga. 

\begin{lemma}\label{L:nonisolating}
Let us assume that \bga is not an internal congruence of $L$.
Then $\bga = \one$.
\end{lemma}

\begin{proof}
Indeed, if \bga is not an internal congruence of $L$, 
then there is an $x \in L - \set{o, i}$ 
such that $\cng x = o (\bga)$
or $\cng x = i (\bga)$. Using the sublattice $A$
provided by Lemma~\ref{L:big}, 
we conclude that $\bga = \one$,
since $A$ is a simple $\set{o,i}$-sublattice.
\end{proof}

\subsection{The congruences of a W lattice}\label{S:Wcongs}
We start with the congruences of the lattice $W(p, q)$ 
with $p < q \in P$, see Figure~\ref{F:Wcongs}.

\begin{lemma}\label{L:Scongs}
The lattice $W(p, q)$ has two internal congruences:
\[
   \con{a_p, b_p} < \con{a_q, b_q}.
\]
\end{lemma}

\begin{figure}[bth]
\centerline{\includegraphics[scale=0.9]{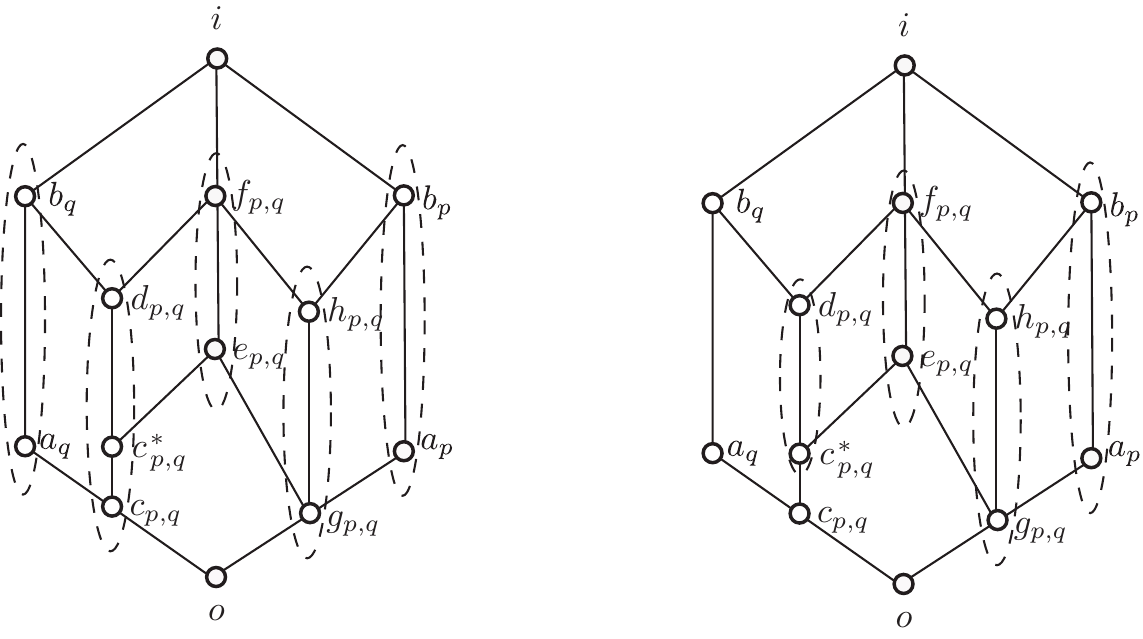}}
\caption{The internal congruences of $W(p, q)$
for $p < q \in P$}\label{F:Wcongs}
\end{figure}

\begin{proof}

An easy computation. 

First, check that Figure \ref{F:Wcongs} correctly
describes the two join-irreducible internal con\-gruences
$\con{a_p, b_p}$ and $\con{a_q, b_q}$.

Then, check all 19 prime intervals $[x, y]$
and show that $\con{x, y}$ is either not an internal con\-gruence 
or equals $\con{a_p, b_p}$ or $\con{a_q, b_q}$.
For instance, \[\con{e_{p,q}, f_{p,q}} = \con{a_p, b_p}\]
and $\con{h_{p,q}, b_p}$ is not an internal con\-gruence because  
\[\cng f_{p,q}=i(\con{h_{p,q}, b_p}).\]
The other 17 cases are similar. 
 
Finally, note that the two 
join-irreducible internal congruences 
we found are comparable, 
so there are no other internal congruences.
\end{proof}

\subsection{The congruences of flag lattices}\label{S:flagcongs}

There is only one important congruence of a flag lattice.
It is $\con{a_p, b_p} = \con{c_i, c_{i+1}}$. 
It identifies $\con{a_p, b_p}$ with $\con{c_i, c_{i+1}}$,
where $[c_i, c_{i+1}]$ is of color~$p$.

Note that there may be many flag lattices 
containing a given $a_p, b_p$.

\subsection{The congruences of $L$}
Let $\bga$ be an internal congruence of $L$. 
Then for every $A \in \CompLat C$,
we associate with $\bga$ the internal congruence $\bga_A$ of $A$,
the restriction of $\bga$ to $A$.
These congruences are compatible, in the following sense.

Let $A, B \in \CompLat C$,
let $\bgb$ be internal congruences on $A$,
and let $\bgg$ be internal congruence on $B$. 
We call the congruences $\bgb$ and $\bgg$ \emph{compatible}, 
if either $A$ and $B$ are not adjacent 
or they are adjacent and $\bgb_{A \ii B} = \bgg_{A \ii B}$.

\begin{lemma}\label{L:comp}
Let $\bga(A)$ be an internal congruence for 
every $A \in \CompLat C$.
Let us assume that the congruences 
$\bga(A)$ are compatible.
Then there is an internal congruence $\bga$ of $L$
such that $\bga(A) = \bga_A$ for $A \in \CompLat C$.
This congruence $\bga$ is unique.
\end{lemma}

\begin{proof} 
By compatibility, we can define a binary relation $\bga$ 
as the union of the $\bga(A)$, that is,
\[
   \bga = \UUm{\bga(A)}{A \in \CompLat C}.
\]
By definition, $\bga$ is reflexive and transitive. 
To prove that $\bga$ is a congruence, 
it is sufficient to verify the Substitution Properties.
The Meet Substitution Property is trivial.
By utilizing the Technical Lemma, see \cite[Lemma~11]{LTF}
and \cite[Theorem 3.1]{CFL2}, we only have to do the 
Join Substitution Property for two comparable elements.

So let $a \leq b \in A - \set{o,i}$, where $A$ is
a component lattice of $L$, let $c \in L - \set{o,i}$,
and let $\cng a = b (\bga)$.
By Corollaries~\ref{C:compl} and~\ref{C:adjacent},
we can assume that there is a component lattice $B$ of $L$
such that $c \in B$.  
Let $U = A \ii B$ be the chain shared by $A$ and $B$. 
We assume also that $a \jj c, b \jj c \in L - \set{o,i}$.
Since $b \jj c \neq i$, it follows that $a^U, b^U, c^U < i$.
Therefore, by Corollary~\ref{C:adjacent}, 
$a \jj c, b \jj c \in U - \set{i} \ci A$,
and the congruence $\cng a = b (\bga)$ now follows
because it holds in $A$ for $\bga(A)$.
(By~utilizing the Technical Lemma for Congruences of Finite Lattices, see \cite{gG14b},
we could assume that $a \prec b,c$ in the last paragraph.)
\end{proof}

Now we are ready to describe 
the join-irredcucible congruences of $L$.
Of course, the unit congruence, $\one$, is join-irreducible, generated by any nontrivial interval 
$[0, x]$ and $[x, 1]$, as well as by $[a_1, b_1]$.

\begin{definition}\label{D:def}
Let $r \in P$ with $r < 1$. 
For $A \in \CompLat C$, define $\bgr(A)$ as follows.

\begin{enumeratei}
\item Let $A = W(p,q)$ for $p < q \in P$. Define
\begin{equation}\label{E:ro}
   \bgr(W(p,q)) = 
       \begin{cases}
       \consub{W(p,q)}{a_q,b_q}   
       &\text{for $q \leq r$};\\
       \consub{W(p,q)}{a_p,b_p}   
       &\text{for $q \nleq r$ and $p \leq r$};\\
       \zero           &\text{for $q \nleq r$ and $p \nleq r$}.\\
       \end{cases}
\end{equation}
\item\label{case1} Let $A = \Flag(c_i)$ for $i < n$ with 
$\col [c_i, c_{i+1}] = p$.
Now we define the congruence $\bgr(\Flag(c_i))$
by enumerating the prime intervals 
it collapses in $\Flag(c_i)$\tup{:}
\begin{enumerate}
\item[\tup{(a)}] the five prime intervals of $\con{a_p, b_p}$ 
in $\Flag(c_i)$ provided that $p \leq r$, 
see Figure~\ref{F:Flagcong};
\item[\tup{(b)}] the prime intervals $[c_j, c_{j+1}] \in C$ 
satisfying $\con{c_j, c_{j+1}} \leq r$;
\item[\tup{(c)}] the prime intervals $[c'_j, c'_{j+1}]$ 
with $j < i$ satisfying $\con{c_j, c_{j+1}} \leq r$.
\end{enumerate}
\item Let $A = S$. Then $\bgr(S) = \zero$.
\item Let $A = C_u$. Then $\bgr(C_u) = \zero$.
\end{enumeratei}
\end{definition}

\begin{lemma}\label{L:jicong}
The congruences $\setm{\bgr(A)}{A \in \CompLat C}$ are compatible.
\end{lemma}
\begin{figure}[bth]
\centerline{\includegraphics[scale=1]{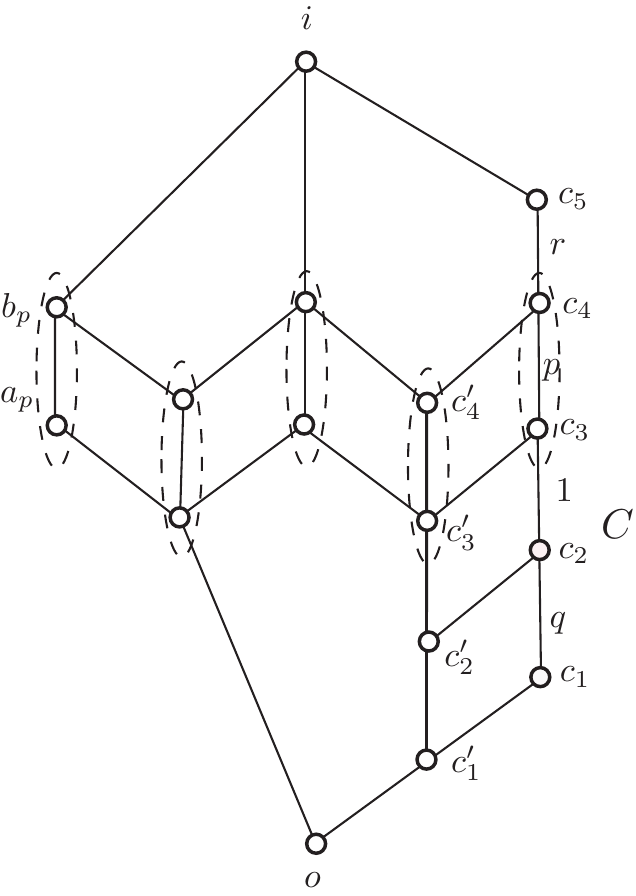}}
\caption{A congruence of $\Flag(c_3)$}\label{F:Flagcong}
\end{figure}

\begin{proof}
Let $A \neq B \in \CompLat C$.
Let $[x, y]$ be a prime interval of $L$ satisfying 
\begin{align}
   &0 < x \prec y < i,\label{E:cov}\\
   &[x, y] \ci A \ii B,\label{E:shared}\\
   &\cng x = y ({\bgr(A)}).\label{E:sharedcong}
\end{align}
It follows from \eqref{E:shared} that both $A$ and $B$
are W lattices or flag lattices. 
We distinguish two cases.

\emph{Case 1.} Let $A$ be the flag lattice $\Flag(c_i)$, 
where $i < n$ and $\col [c_i, c_{i+1}] = p$. 
The only prime interval in $A$ with \eqref{E:cov}
that can be shared with another W lattice or flag lattice
is $[a_p, b_p]$ and it only can be shared with
the W lattice $B = W(p, q)$ for $p < q \in P$ or 
with the W lattice $B = W(q, p)$ for $q < p \in P$.

$\bgr(A)$ restricted to $[x, y]$ is $\zero$ if $p \nleq r$ 
and $\one$ if $p \leq r$ by Definition~\ref{D:def}(i).
Similarly, $\bgr(B)$ restricted to $[x, y]$ 
is $\zero$ if $p \nleq r$ 
and $\one$ if $p \leq r$ by Definition~\ref{D:def}(ii).
So we get compatibility.

\emph{Case 2.} Let $A$ be the W lattice 
$W(p, q)$ for $p < q \in P$. 
The only prime interval in $A$ with \eqref{E:cov} 
and \eqref{E:sharedcong}
that can be shared with another W lattice
is $[a_p, b_p]$ and it only can be shared with
the W lattice $B = W(p, q)$ for $p < q \in P$ or 
with the W lattice $B = W(q, p)$ for $q < p \in P$.
(Sharing it with a flag lattice was discussed in Case 1.)
In both cases we apply \eqref{E:ro} to get compatibility:
if $p \leq r$, then we get $[a_p, b_p]$ collapsed 
by both $\bgr_A$ and $\bgr_B$; 
if $p \nleq r$, then $\bgr_A = \zero$ and $\bgr_B = \zero$.
\end{proof}

By Lemma~\ref{L:comp}, we have the congruence 
$\bgr = \con{a_r, b_r}$ on $L$ so that 
$\bgr_A = \bgr(A)$ for every component lattice $A$ of $L$,

\begin{corollary}\label{C:jicong}
The join-irreducible congruence of $L$ are the congruences
$\con{a_r, b_r}$ for $r \in P$.
\end{corollary}

\begin{corollary}\label{C:allcong}
The map
\[
   r \mapsto \con{a_r, b_r} \text{ for $r \in P$}
\]
uniquely extends to an isomorphism $\gf$ 
between $D$ and $\Con L$.
\end{corollary}

\subsection{Principal congruences of $L$}\label{S:princecongs}
To complete the proof of Theorem~\ref{T:main2},
it remains to prove the following two results.

\begin{lemma}\label{L:laststep1}
Under the isomorphism~$\gf$ 
of Corollary~\ref{C:allcong}, 
if $q \in Q \ci D$, then $\gf q$ is a principal congruence of $L$.
\end{lemma}

\begin{proof}
Let $q \in Q$. Since $Q$ is represented
by the chain $C$ colored by $P = \J D$, there is an interval
$[x,y]$ of $C$ such that $q = \JJ \colset[x,y]$.
So $\gf$ maps $q \in Q$ to a principal congruence $\con{x,y}$ of $L$.
\end{proof}

\begin{lemma}\label{L:laststep2}
Let $d \in D$ and let $\gf d$ be a principal congruence of $L$.
Then $d \in Q$.
\end{lemma}

\begin{proof}
Let $d \in D$ and let $\gf d = \con{x,y}$, 
where $x \leq y \in L$. If $x = y$, then $d = 0 \in Q$.
If $x \prec y$ in $L$, then $\gf d = \con{a_p,b_p}$,
where $p \in P \ci Q$.

Finally, let $[x, y]$ be of length at least $2$, that is,
$x = z_0 \prec z_1 \prec \dots \prec z_k = y$, where $2 \leq k$.
By the construction of $L$, there is a 
component lattice $A$ of $L$, such that $x, y \in A$.
If $A$ is a W lattice, or $S$, or $C_u$, 
then $\con{z_i, z_{i+1}} = \one$, for some $i < k$,
so $d = 1 \in P \ci Q$.

Finally, let $A$ be a flag lattice $\Flag(c_i)$ for $i < n$.
By inspecting the diagram of~$\Flag(c_i)$ 
(Figure~\ref{F:Flag} and Figure~\ref{F:Flagcong}),
we conclude that one of the following three cases occurs:

\begin{enumeratei}
\item $\con{z_i, z_{i+1}} = \one$ for some $i < k$;
\item $x, y \in C$;
\item there are elements $u, v \in C$ 
so that $x = u'$ and $y = v'$.
\end{enumeratei}

In Case (i), it follows that $d = 1 \in P \ci Q$.

In Case (ii), we conclude that 
\[
   d = \JJ \colset[x, y]
\]
and so $d \in Q$ by the definition of chain representability,
see Section~\ref{S:repr}.

In Case (iii), we argue as in Case (ii) 
with the interval $[u, v]$.
\end{proof}

\subsection{Discussion}\label{S:Discussion}

Problem 22.1 of my book \cite{CFL2} 
(see Section~\ref{S:Background}) remains unsolved.

In G. Gr\"atzer and H. Lakser~\cite{GLa}, 
we proved some relevant results:
\begin{enumeratei}
\item For a finite distributive lattice $D$,
the set $Q = D$ is representable.
\item If a finite distributive lattice $D$
has a join-irreducible unit element,
then $Q =  \J D \uu \set{0,1}$ is representable.
\item Let $D$ be the eight-element Boolean lattice
with atoms $a_1, a_2, a_3$. 
Then the set $Q = \set{0,a_1, a_2, a_3, 1} \ci D$ 
is not representable.
\end{enumeratei}

We also introduced in G. Gr\"atzer and H. Lakser~\cite{GLa}
the following concept.
Let us call a finite distributive lattice $D$ 
\emph{fully representable}, if every $Q \ci D$ 
is representable provided that $\set{0,1} \uu \J D \ci Q$.
In G. Gr\"atzer and H. Lakser \cite{GLa},
we observe that every fully representable 
finite distributive lattice is planar.

G. Cz\'edli~\cite{gCb} and~\cite{gCc} combine to give 
a deep characterization 
of fully representable finite distributive lattices as follows:

A finite distributive lattice $D$ is 
fully principal congruence representable if{}f
$D$ is planar and it has at most one join-reducible dual atom.

\end{document}